\documentclass[11pt,a4paper,draft,leqno,twoside]{article}
\usepackage{amsfonts,amsmath,amssymb,amsthm,amsxtra,color,graphicx,latexsym,mathtools,xcolor}
\usepackage[top=3cm, bottom=3cm, left=3cm, right=3cm]{geometry}

\newtheorem{definition}{Definition}
\newtheorem{proposition}[definition]{Proposition}
\newtheorem{remark}[definition]{Remark}
\newtheorem{corollary}[definition]{Corollary}

\numberwithin{equation}{section}
\numberwithin{definition}{section}

\newcommand{\Iff}{if\textcompwordmark f}

\newcommand{\R}[1]{\ensuremath{\mathbb{#1}}}
\newcommand{\rn}{\ensuremath{\left(\varPhi,\vect{y}\right)}}
\newcommand{\rns}{\ensuremath{\left(\varPhi^*,\vect{y}\right)}}

\newcommand{\vect}[1]{\boldsymbol{#1}}
\newcommand{\TU}[1]{\textup{#1}}
\DeclareMathOperator{\RANK}{rank}
\DeclareMathOperator{\DIFF}{d\!}

\begin{document}
\title{Generalization of two Bonnet's Theorems\\ to the relative Differential Geometry\\ of the 3-dimensional Euclidean space}

\author{{Stylianos Stamatakis, Ioannis Kaffas and Ioannis Delivos}\\ \emph{Aristotle University of Thessaloniki}\\ \emph{Department of Mathematics}\\ \emph{GR-54124 Thessaloniki, Greece}\\  \emph{e-mail: stamata@math.auth.gr}}
\date{}
\maketitle

\begin{abstract}
\noindent This paper is devoted to the 3-dimensional relative differential geometry of surfaces. In the Euclidean space $\R{E} ^3 $ we consider a surface $\varPhi 
$ with position vector field $\vect{x}$, which is relatively normalized by a relative normalization $\vect{y}
$. A surface $\varPhi^*
$ with position vector field $\vect{x}^* = \vect{x} + \mu \, \vect{y}$, where $\mu$ is a real constant, is called a relatively parallel surface to $\varPhi$. Then $\vect{y}$ is also a relative normalization of $\varPhi^*$.
The aim of this paper is to formulate and prove the relative analogues of two well known theorems of O.~Bonnet which concern the parallel surfaces (see~\cite{oB1853}).


\end{abstract}

\section{Preliminaries}\label{Section1}
In this section  we briefly review the main  definitions, formulae and results on relative differential geometry and fix our notations. For details the reader is referred to \cite{pS62, rW35}.

In the three-dimensional Euclidean space $\R{E} ^3 $ let $\varPhi$ be a $C^r$-surface, $r\geq3$, defined on a region $U$ of $\R{R} ^2$, by an injective $C^{r}$-immersion  $\vect{x} = \vect{x}(u^1,u^2)$,   whose Gaussian curvature 
$\widetilde{K}$ never vanishes.
Let
                \begin{equation*}
                I = g_{ij}\DIFF u^i \DIFF u^j, \quad II = h_{ij}\DIFF u^i \DIFF u^j, \quad III = e_{ij}\DIFF u^i \DIFF u^j, \quad i,j = 1,2,
                \end{equation*}
be the first, second and third fundamental forms of $\varPhi$ and $g^{(ij)}$, $h^{(ij)}$ and $e^{(ij)}$
the inverse of the tensors $g_{ij}$, $h_{ij}$ and $e_{ij}$.

We denote by $\partial_i f$, $\partial_j\partial_i f$ etc. the partial derivatives of a function (or a vector-valued function) $f$ with respect to $u^i$, $i = 1,2$.
A $C^{s}$-relative normalization of $\varPhi$ is a $C^s$-mapping $\vect{y} = \vect{y}(u^1,u^2)$,  $r > s \geq 1$, defined on $U$, such that
            \begin{equation}                \label{1.1}
            \RANK \big( \big\{ \partial_1 \vect{x},\partial_2 \vect{x},\vect{y}\big\}\big) = 3, \quad
            \RANK \big(\big\{\partial_1 \vect{x},\partial_2 \vect{x},\partial_i \vect{y}\big\}\big) = 2, \quad i = 1, 2,
            \end{equation}
for all $\left(u^{1},u^{2}\right) \in U$.
We will say that the pair \rn{} is a relatively normalized surface in $\mathbb{E}^3$.
The line issuing from a point $P \in \varPhi$ in the direction of $\vect{y}$ is called the relative normal of \rn{} at $P$.
When we move the vectors $\vect{y}$ to the origin, the endpoints of them describe the relative image of \rn.

The contravariant vector $\vect{X}$ of the tangent vector space is defined by%
            \begin{equation}                \label{1.5}
            \langle \vect{X},\partial_i \vect{x} \rangle = 0, \quad
            i = 1,2, \quad \text{and}\quad \langle \vect{X},\vect{y} \rangle= 1,
            \end{equation}
where $\langle \; ,\; \rangle$ denotes the standard scalar product in $\mathbb{E}^3$.
The 
quadratic differential form
            \begin{equation*}                
            G \coloneqq - \langle \DIFF \vect{x}, \DIFF \vect{X}  \rangle
            \end{equation*}
is called the relative metric of \rn{} and is definite or indefinite, depending on whether the Gaussian curvature $\widetilde{K}$ is positive or negative, respectively.
Let
            \begin{equation*}               
            G_{ij} = - \langle  \partial_i \vect{x}, \partial_j \vect{X} \rangle
                   = - \langle  \partial_j \vect{x}, \partial_i \vect{X} \rangle = G_{ji}
             \end{equation*}
be the coefficients of the relative metric. Obviously
            \begin{equation*}               
            G_{ij} = \langle  \partial_j \partial_i \vect{x}, \vect{X} \rangle.
            \end{equation*}
From now on we shall use $G_{ij}$ as the fundamental tensor for ``raising and lowering the indices''  in the sense of the classical tensor notation.

Let $\vect{\xi}$ denote the (Euclidean) unit normal vector to $\varPhi$.
By virtue of ~ \eqref{1.1} the support function of Minkowski $q(u^1,u^2)$ of the relative normalization $\vect{y}$, which is defined by
            \begin{equation}                \label{1.15}
            q\coloneqq \langle \vect{\xi},\vect{y} \rangle
            \end{equation}
never vanishes on $U$.
Because of~\eqref{1.5}, it is
            \begin{equation}                \label{1.20}
            \vect{X} = q^{-1} \, \vect{\xi}, \quad  G_{ij} = q^{-1} \, h_{ij},  \quad G^{\left( ij \right)} = q\, h^{\left( ij \right)},
            \end{equation}
where $G^{( ij) } $ is the inverse of the tensor $G_{ ij }$.

We mention that, when a nonvanishing $C^s$-function $q(u^1,u^2)$ is given, then there exists a unique  relative normalization $\vect{y}$, which is determined by
            \begin{equation}                \label{1.25}
            \vect{y}=-\nabla^{II}\!\!\left( q,\, \vect{x} \right) + q \,\vect{\xi},
            \end{equation}
where $\nabla^{II}$ denotes the first Beltrami-operator with respect to $II$, such that the support function of the relative normalization \eqref{1.25} is the given function $q$, see~ \cite[p.197]{fM89}.

Let $\nabla^{G}_{i}$ denote the covariant derivative in the direction $u^i$ corresponding to $G$ and
            \begin{equation*}                
            A_{ijk}\coloneqq \langle \vect{X},\,\nabla^{G}_{k} \, \left(\nabla^{G}_{j}\,\partial_i \vect{x}\right) \rangle
            \end{equation*}
the (symmetric) Darboux-tensor.
By means of it the Tchebychev-vector $\vect{T}$ of the relative normalization $\vect{y}$
            \begin{equation*}                
            \vect{T}\coloneqq T^{m}\, \partial_m \vect{x},\quad \text{where\quad } T^{m}\coloneqq \frac{1}{2}A_{i}^{im}
            \end{equation*}%
is defined. The quadratic differential form
            \begin{equation*}                
            B \coloneqq \langle \DIFF \vect{y}, \DIFF \vect{X}  \rangle
            \end{equation*}
is the relative shape operator
and has the coefficients $B_{ij}$ such that
            \begin{equation}                \label{1.45}
            B_{ij} = \langle  \partial_i \vect{y}, \partial_j \vect{X}\rangle
            = \langle  \partial_j \vect{y}, \partial_i \vect{X}\rangle
            = B_{ji}.
            \end{equation}
Obviously
            \begin{equation*}                
            B_{ij} = - \langle  \partial_j \partial_i \vect{y}, \vect{X}\rangle
            = - \langle  \vect{y},\partial_j \partial_i \vect{X} \rangle.
            \end{equation*}
Then the following Weingarten type equations  are valid
            \begin{equation}                \label{1.50}
            \partial_i \vect{y} = -B_{i}^{j} \,  \partial_j  \vect{x}.
            \end{equation}%
Special mention should be made of the equiaffine normalization $\vect{y}_{AFF}$, which, on account of \eqref{1.25}, is defined by means of the equiaffine support function
                 \begin{equation}                        \label{4.12}
                 q_{AFF} \coloneqq |\widetilde{K}|^{1/4}.
                 \end{equation}
The equiaffine normalization and its homothetics, i.e., the normalizations which are constantly proportional to $\vect{y}_{AFF}$, are characterized by the vanishing of the corresponding  Tche\-by\-chev-vec\-tor.

The eigenvalues $\kappa_1$ and $\kappa_2$ of the tensor $B_{ij}$ are called relative principal curvatures of \rn{}. Their reciprocals $R_1$  (when $\kappa_1 \neq 0$) and $R_2$ (when $\kappa_2 \neq 0$) are said to be the relative principal radii of curvature.
The product $K$ and the sum $2H$ of the relative principal curvatures are called relative curvature and relative mean curvature, respectively.
They are computed by means of the formulae
            \begin{align}
            \kappa_1 \cdot \kappa_2      & = K   = \det \left(B_{i}^{j}\right),             \label{1.55}\\
            \kappa_1 + \kappa_2          & = 2H  = B_{i}^{i}.                               \label{1.60}
            \end{align}
We consider the relative lines of curvature of \rn.
They are the curves of $\varPhi$ which are characterized by the property that the relative normals along them form developable surfaces.
Whenever at least one of the surfaces $\varPhi$ and its relative image have positive Gaussian curvature, the relative lines of curvature are real and in this case
            \begin{equation*}                    
            H^2 - K \geq 0,
            \end{equation*}
holds true, see ~\cite[p.~433]{rW35}. In the rest of this paper we assume that the relative lines of curvature are real.
Then they are determined by the differential equation
                \begin{equation}                        \label{1.70}
                B_1^2 \left( \DIFF u^1 \right)^2 + \left( B_2^2 -  B_1^1\right) \DIFF u^1 \! \DIFF u^2 - B_2^1 \left(\DIFF u^2\right)^2 = 0,
                \end{equation}
cf.~\cite[p.~30]{fM82}.
Finally, the surfaces
            \begin{equation}                         \label{1.75}
            \varPhi_i \colon \vect{x}_i = \vect{x} + R_i \, \vect{y} , \,\,\, 1 = 1, 2,
            \end{equation}
are the loci of the edges of regression of the developable surfaces consisting of the relative normals along the relative lines of curvature. They are called relative centre surfaces of \rn.
\section{Relatively parallel surfaces in $\R{E}^3$}\label{Section2}
We consider a surface $\varPhi ^*$ which is parametrized by
                \begin{equation}                     \label{2.1}
                \vect{x}^* = \vect{x} + \mu \, \vect{y},
                \end{equation}
where $\mu = \mu(u^1,u^2)$ is a $C^s$-function. 
To the point $P(u^1_0,u^2_0)$ of $\varPhi$ corresponds the point $P^*(u^1_0,u^2_0)$ of $\varPhi^*$ so that their position vectors are $\vect{x}(u^1_0,u^2_0)$ and  $\vect{x}^*(u^1_0,u^2_0)$, respectively.
We first prove the following
                \begin{proposition}\label{prop:2.1}
                The surfaces $\varPhi$ and $\varPhi^*$ have parallel normals at corresponding points \Iff{} \/ $\mu = const$.
                \end{proposition}
                \begin{proof}
                We assume that the surfaces $\varPhi$ and $\varPhi^*$ have parallel normals at corresponding points. Then $\langle \partial_i \vect{x}^*,\vect{\xi} \rangle = 0$ is valid for $i = 1$, $2$, and so by means of \eqref{2.1}
                \[
                \langle \partial_i \vect{x} + \left( \partial_i \mu\right) \,\vect{y} + \mu \, \partial_i \vect{y},\vect{\xi} \rangle = 0.
                \]
                On account of \eqref{1.15} and \eqref{1.50} we find $q \, \partial_i\mu = 0$, $i = 1$, $2$, and because of $q \neq 0$ we obtain $\partial_i\mu = 0$, hence $\mu = const$. One can easily verify that the converse is valid too.
                \end{proof}
We consider a constant $\mu \neq 0$. It is readily verified from \eqref{1.50}, \eqref{1.55} and \eqref{1.60} that the vector product of the partial derivatives $\partial_i \vect{x}^* $ satisfies the relation
                \begin{equation}                     \label{2.5}
                \partial_1 \vect{x}^*  \times \partial_2  \vect{x}^* =
                A \; \partial_1 \vect{x} \times \partial_2  \vect{x},
                \end{equation}
where $A = A(u^1,u^2)$ is the  function
                \begin{equation}                    \label{2.10}
                A\coloneqq \mu^2  K - 2 \mu \,  H + 1.
                \end{equation}
When the relative curvature $K$ of \rn{} does not vanish, $A $ can be written by means of the relative principal radii of curvature $R_1$ and $R_2$ of \rn{} as follows
                \begin{equation*}                    
                A = K \left(\mu - R_1 \right) \left(\mu - R_2\right).
                \end{equation*}
In what follows we suppose that $A \neq 0$ everywhere on $U$, hence, in case $K \neq 0$, it is  $\mu \neq R_1, R_2$. Then the parametrization  \eqref{2.1} of $\varPhi ^*$ is regular. Equation \eqref{2.5} also shows that the unit normal vector $\vect{\xi^*}$ at a point $P^*$ of  $\varPhi ^*$ can be chosen to be  the same with the unit normal vector $\vect{\xi}$  of $\varPhi$ at the corresponding point $P$ of $\varPhi$.
                 We call the surface $\varPhi ^*$ a relatively parallel surface of $\varPhi$. Throughout what follows, we shall freely use for $\mu$  the expression ``relative distance". 

It is obvious that the relations \eqref{1.1} are as well  valid if the parametrization $\vect{x}(u^1,u^2)$ of $\varPhi$ is replaced by the parametrization $\vect{x}^*(u^1,u^2)$ of $\varPhi ^*$, cf. \eqref{2.1}.
Therefore $\vect{y}$ is a relative normalization for $\varPhi ^*$ and so the expression ``relatively parallel'' is fully justified.
There exist infinitely many relatively parallel surfaces \rns{} of a given surface \rn{}, each of which corresponds to a different value of $\mu$. All of them possess the same relative image with \rn.

By means of~\eqref{1.15}, (\ref{1.20}a) and ~ \eqref{1.45}, it is clear that \emph{the relatively parallel surfaces} \rn{} \emph{and} \rns{} \emph{have in common}\\
(a) \emph{the support function} $q$, \\
(b) \emph{the covector} $\vect{X}$ \emph{of the tangent vector space and} \\
(c) \emph{the covariant coefficients of their shape operator, i.e,}
                \begin{equation*}                    
                B_{ij} = B^*_{\phantom{^*}ij}.
                \end{equation*}
Property (c) generalizes a result of the equiaffine surface theory \cite[p.~174]{pS62}.
                \begin{remark}
                If the relative normalization $\vect{y}$ of $\varPhi$ is the Euclidean one \TU{(}$\vect{y} = \vect{\xi}$\TU{)}, then $q = 1$ and vice versa \TU{(}cf.~ \eqref{1.15} and \eqref{1.25}\TU{)}. In this case the concept of the relatively parallel surfaces reduces to the Euclidean one.
                \end{remark}
We consider two relatively parallel surfaces \rn{} and \rns{} and we mark with an asterisk the functions which result if $\vect{x}(u^1,u^2)$ is replaced by $\vect{x}^*(u^1,u^2)$ in the above equations. Furthermore we refer by (j*) to the formula, which results from formula (j) after this replacement.

From ~(\ref{1.20}b), \eqref{1.50} and ~\eqref{2.1} it follows
                \begin{equation}                        \label{3.1}
                h_{\phantom{^*}ij}^* = h_{ij} - \mu \,q \,B_{ij}, \quad	i,j = 1,2.
                \end{equation}
By combining (\ref{1.20}b), (\ref{1.20}*b) and \eqref{3.1} we get
                \begin{equation*}                     \label{3.5}
                G_{\phantom{^*}ij}^* = G_{ij} - \mu \,B_{ij}, \quad	i,j = 1,2.
                \end{equation*}
On account of \eqref{1.50}, (\ref{1.50}*) and \eqref{2.1}, we find
                \begin{equation*}
                B_{\phantom{^*}i}^{*j} - \mu \, B_{\phantom{^*}i}^{*k} \, B_k^j - B_i^j = 0 , \quad i,j = 1,2.
                \end{equation*}
Solving this system we get the mixed components of the shape operator of \rns{}
                \begin{equation}                           \label{3.20}
                B_{\phantom{^*}1}^{*1} = \frac{B_1^1 - \mu \, K}{A}, \quad B_{\phantom{^*}1}^{*2}  = \frac{B_1^2}{A}, \quad
                B_{\phantom{^*}2}^{*1} = \frac{B_2^1}{A},            \quad B_{\phantom{^*}2}^{*2}  = \frac{B_2^2 - \mu \, K}{A}.
                \end{equation}
By substitution in (\ref{1.55}*) and (\ref{1.60}*), we obtain for the relative curvature $K^*$ and the relative mean curvature $H^*$ of \rns{}
                \begin{align}
                K^* & = \frac {K} {A},                    \label{3.25} \\
                H^* & = \frac{H - \mu K} {A}.             \label{3.30}
                \end{align}
From \eqref{1.50} and \eqref{2.1} we find
                    \begin{equation}                    \label{3.45}
                    g^*_{\phantom{^*}ij} = g_{ij} - \mu \left( B_i^r \, g_{rj} + B_j^r \, g_{ri} \right) + \mu ^2 \,B_i^r \, B_j^s \, g_{rs}.
                    \end{equation}
Now, we consider two relatively parallel surfaces of \rn{} at relative distances $\mu_1$ and $\mu_2$, respectively.
By means of \eqref{3.25} and \eqref{3.30} we see that \emph{these surfaces have common}

(a)  \emph{relative curvature \Iff{}}\emph{}
                \begin{equation*}
                K = 0 \quad \text {or} \quad \mu_1 + \mu_2 = \frac{2H}{K} \,\, \text{and} \,\,K \neq 0,
                \end{equation*}
(b)  \emph{relative mean curvature \Iff{}}
                \begin{equation*}                       
                KH\left( \mu_1 + \mu_2\right) - K^2\, \mu_1  \mu_2 + K - 2H^2 = 0.
                \end{equation*}
Obviously \emph{a relatively parallel surface of} \rn{} \emph{at relative distance} $\mu$ \emph{has}

(a) \emph{the same relative curvature with} \rn{} \emph{\Iff{}}
                \begin{equation*}                        
                K = 0 \quad \text {or} \quad \mu = \frac{2H}{K} \,\, \text{and} \,\, K \neq 0,
                \end{equation*}
(b) \emph{the same  relative mean curvature with} \rn{}  {\Iff{}}
                \begin{equation*}                     
                H \neq 0, \,\, K \neq 0 \,\,\text{and} \,\, \mu = \frac{2H^2 - K}{KH} \quad \text{or} \quad H = 0 \,\, \text{and} \,\, K = 0.
                \end{equation*}
From (a) it is apparent that if \rn{} is a relative minimal surface with $K \neq 0$, then there is no relative minimal surface relatively parallel to it.

\section{The main results}\label{Section3}

We  are now in position to state the main results of this paper which generalize two well known theorems of O.~Bonnet in the classical differential geometry of surfaces in $\R{E}^3$, see ~ \cite{oB1853}.
                \begin{proposition}                         \label{Prop:BonnetK}
                Let \rn{} be a relatively normalized surface  of constant positive relative curvature $K$ in the Euclidean space $\mathbb{E}^3$. Then there are two relatively parallel surfaces to \rn{} which have constant relative mean curvature.
                \end{proposition}
                \begin{proof}
                We consider  the relatively parallel surfaces of \rn{} at relative distances $\mu = \mp 1 / \sqrt{K}$ (for $H \neq \mp \sqrt{K}$). Substitution of these values of $\mu$ in \eqref{3.30} gives
                \begin{equation*}                            
                H^* = \pm \frac{\sqrt{K}}{2} = const. \qedhere
                \end{equation*}
                \end{proof}
                \begin{remark}
                By using \eqref{3.30} one can easily confirm that if $H$ is not constant, then there are   exactly two such surfaces.
                \end{remark}
\begin{proposition}                         \label{Prop:BonnetH}
                Let \rn{} be a relatively normalized surface  of nonvanishing constant relative  mean curvature $H$ in the   Euclidean space $\mathbb{E}^3$. Then there is
                one relatively parallel surface to \rn{} of constant relative curvature and another one of constant relative mean curvature.
                \end{proposition}
                \begin{proof}
                The proof follows that of the preceding Proposition. By considering  the relatively parallel surfaces of \rn{} at relative distances $\mu = 1 / 2H$ (for $K \neq 0$) and $\mu = 1/H$ (for $H^2 \neq K$), respectively, we find from \eqref{3.25} and \eqref{3.30}
                \begin{equation*}                            
                K^* = 4H^2 \quad \text{and} \quad H^* = -H. \qedhere
                \end{equation*}
                \end{proof}
                \begin{remark}
                By using \eqref{3.25} or \eqref{3.30} one can easily confirm that if $K$ is not constant, then there is exactly one relatively parallel surface of constant relative curvature and exactly one of constant relative mean curvature.
                \end{remark}
                \begin{remark}
                For $q = 1$ arises the Euclidean normalization and Proposition \ref{Prop:BonnetK} and the first part of Proposition \ref{Prop:BonnetH} are the original Bonnet's  theorems.
                \end{remark}
Other propositions of this type are the following
                \begin{proposition}                         \label{Prop:H/K}
                Let \rn{} be a relatively normalized surface of constant sum $R_1 + R_2$ of its relative principal radii of curvature in the   Euclidean space $\mathbb{E}^3$, which satisfies $K \neq 0$ and $H^2 - K \neq 0$. Then every relatively parallel surface to \rn{} has also constant sum of its relative principal radii of curvature and there is exactly one relatively parallel surface to \rn{} which is relatively minimal. 
                \end{proposition}
                \begin{proof}
                From \eqref{3.25} and \eqref{3.30} we result
                \begin{equation*}                            
                \frac{H^*} {K^*} = \frac{H} {K} - \mu,
                \end{equation*}
                or, equivalently,
                \begin{equation}                            \label{4.25}
                R_{\phantom{^*}1}^* + R_{\phantom{^*}2}^*  = R_{1} + R_{2} - 2 \mu.
                \end{equation}
                Consequently
                \begin{equation*}                            \label{4.30}
                R_{1} + R_{2} = 2c_1 = const. \iff R_{\phantom{^*}1}^* + R_{\phantom{^*}2}^* = 2 c_2 = const., \quad \text {where} \quad c_2 = c_1 - \mu.
                \end{equation*}
                As long as $ R_{1} + R_{2} = 2c_1 = const$., by choosing $\mu = c_1 $, we find from \eqref{4.25} that $H^* = 0$.
                Thus, among all the relatively parallel surfaces to \rn{}, only the one at relative distance $\mu = H/K$ from $\varPhi$ is a relative minimal surface.
                \end{proof}
                \begin{proposition}                         \label{Prop:H/K2}
                Let \rn{} be a relatively normalized surface  of constant sum $R_1 + R_2$ of its relative principal radii of curvature in the   Euclidean space $\mathbb{E}^3$, which satisfies $K \neq 0$ and $H^2 - K \neq 0$. Then  there is one relatively parallel surface to \rn{} which has the same relative curvature and opposite mean curvature with it. 
                \end{proposition}
                \begin{proof}
                One gets easily the above results by substitution of $\mu = 2H/K$ in \eqref{3.25} and \eqref{3.30}.
                \end{proof}
\section{Some further results}\label{Section4}
From the relations  \eqref{3.25} and \eqref{3.30} we find by direct computation
                \begin{equation}                           \label{5.1}
                \frac {H^{*2} - K^*} {K^{*2}} = \frac {H^{2} - K} {K^{2}}.
                \end{equation}
Furthermore, by means of \eqref{3.45} and a straightforward calculation, we obtain
                    \begin{equation*}                       
                    \det\left(g^*_{\phantom{^*}ij} \right) = A^2 \,  \det\left(g_{ij} \right).
                    \end{equation*}
Besides, by using \eqref{3.1},  we get
                    \begin{equation*}             
                    \det\left(h^*_{\phantom{^*}ij} \right) = A \,  \det\left(h_{ij} \right).
                    \end{equation*}
From these two last equations we find that the Gaussian curvatures $\widetilde{K}$ and $\widetilde{K}^*$ of $\varPhi$ and $\varPhi^*$, respectively, are related by
                    \begin{equation}                        \label{5.15}
                    \widetilde{K}^* = \frac{\widetilde{K}}{A}.
                    \end{equation}
Taking into account equation \eqref{3.25} we obtain
                \begin{equation*}                           
                \frac{\widetilde{K}^*}{K^*} = \frac {\widetilde{K}} {K}.
                \end{equation*}
We summarize the above results in the form of a proposition:
                \begin{proposition}                        \label{prop:3.1}
                The functions
                \begin{equation*}
                \frac {H^{2} - K} {K^{2}} \quad \text{and} \quad \frac{\widetilde{K}}{K}
                \end{equation*}
                 remain invariant by the transition to anyone of the relatively parallel surfaces of $\varPhi$.
                \end{proposition}
Using   \eqref{3.20} we observe that the differential equations \eqref{1.70} and (\ref{1.70}*) of the relative lines of curvature of \rn{} and \rns{} respectively, are identical and consequently \emph{the relative lines of curvature of every relatively parallel surface} \rn{} \emph{correspond to each other and to the lines of curvature of the initial relatively normalized surface} \rn.
Finally, from \eqref{3.25} and \eqref{3.30} it follows
                \begin{equation*}                           
                R_{\phantom{^*}i}^* = R_i - \mu, \,\,\, i = 1, 2.
                \end{equation*}
Moreover, taking into account  \eqref{1.75} we easily see \emph{that all relative parallel surfaces of \rn{} have the same relative centre surfaces}.

On account of the well known relations
                \begin{equation*}
                e^{(ij)} = h^{(ir)} \, h^{(js)} \, g_{rs}
                \end{equation*}
and the Euclidean Weingarten equations
                \begin{equation*}
                \partial_i \vect{\xi} = -h_{ij} \, g^{(jk)} \, \partial_k \vect{x},
                \end{equation*}
one can immediately verify the following relation
                \begin{equation}                            \label{5.30}
                \nabla^{II}\!\!\left( f,\, \vect{x} \right) = - \nabla^{III}\!\!\left( f,\, \vect{\xi} \right)
                \end{equation}
for a $C^1$-function $f(u^1,u^2)$, where $\nabla^{III}\!\!\left( f,\, \vect{\xi} \right)$ is the first Beltrami-operator with respect to $III$.

We conclude this work by studying the equiaffine normalizations of $\varPhi$ and $\varPhi^*$, where, as before, \rns{} is a relatively parallel surface to the given relatively normalized surface \rn{} at relative distance $\mu$.
The  support functions of the equiaffine normalizations $\vect{y}_{AFF}$ and $\vect{y}^*_{AFF}$ of  $\varPhi$ and $\varPhi^*$, respectively,
are given by \eqref{4.12} and
                \begin{equation}                        \label{5.35}
                q_{AFF}^* \coloneqq |\widetilde{K}^*|^{1/4},
                \end{equation}
respectively. Recalling that $\varPhi$ and $\varPhi^*$ have common Gaussian mapping and in view of \eqref{1.25} and \eqref{5.30} we have
                \begin{align}
                \vect{y}_\TU{AFF}     = \nabla^{III}\!\!\left( q_{AFF},\, \vect{\xi} \right) + q_{AFF} \,\vect{\xi}, \label{5.40} \\
                \vect{y}_\TU{AFF}^*   = \nabla^{III}\!\!\left( q_{AFF}^*,\, \vect{\xi} \right) + q_{AFF}^* \,\vect{\xi}. \label{5.45} \end{align}
We prove the following proposition:
                \begin{proposition}
                The surfaces $\varPhi$ and $\varPhi^*$ have parallel affine normals at corresponding points \Iff{} the relative curvature $K$ and the relative mean curvature $H$ of \rn{} are connected  with a relation of the form
                \begin{equation}                        \label{5.50}
                \mu \, K - 2H = const.
                \end{equation}
                \end{proposition}
                \begin{proof}
                The affine normals  of the surfaces $\varPhi$ and $\varPhi^*$ are parallel \Iff{}
                \begin{equation}                        \label{5.51}
                \vect{y}_\TU{AFF}   = c \,\vect{y}_\TU{AFF}^*, \quad    c \in \R{R}.
                \end{equation}
                On account of \eqref{5.40} and \eqref{5.45}, equation \eqref{5.51} is equivalent to
                \begin{equation*}                       
                q_{AFF} = c \, q_{AFF}^*,
                \end{equation*}
                or by means of \eqref{4.12} and \eqref{5.35}, to
                \begin{equation*}                        
                \left|\frac{\widetilde{K}}{\widetilde{K}^*}\right| = c^4.
                \end{equation*}
                Hence, by using \eqref{5.15}, it turns out that
                \begin{equation*}                     
                A = \pm c^4,
                \end{equation*}
                or because of \eqref{2.10}
                \begin{equation*}                    
                \mu K - 2H = \frac{\pm c^4 - 1}{\mu} = const.
                \end{equation*}
                The argument can be reversed to show sufficiency, completing the proof.
                \end{proof}
                \begin{corollary}
                If there are two relatively parallel surfaces to \rn{} such that their affine normals are parallel to the affine normals of $\varPhi$ at corresponding points, then the relative curvature $K$ and the relative mean curvature $H$ of \rn{} are constant.
                \end{corollary}
Moreover, in the special case where the relative normalization $\vect{y}$ is the equiaffine one, we obtain as corollary the following well known result (see~\cite[p.~147]{dG48} and \cite[p.~144]{pS62}):
                \begin{corollary}
                Let $\left(\varPhi,\vect{y}_{AFF}\right)$ be a surface equipped with the equiaffine normalization and $\varPhi^*$ a relatively parallel to it. Then they have the same affine normals at corresponding points \Iff{} the affine curvature $K_{AFF}$ and the affine mean curvature $H_{AFF}$ of $
                \varPhi
                $ are connected with a relation of the form \eqref{5.50}.
                \end{corollary}

\end{document}